\documentclass[10pt,final]{amsart}
\usepackage[active]{srcltx}
\usepackage[latin1]{inputenc}
\usepackage{amsmath,bbm}
\usepackage{amsfonts}
\usepackage{pgf,tikz}
\usetikzlibrary{arrows}
\usepackage{amssymb}
\usepackage[colorlinks=true,urlcolor=blue,
citecolor=red,linkcolor=blue,linktocpage,pdfpagelabels,bookmarksnumbered,bookmarksopen]
{hyperref}
\usepackage{pgf,tikz}
\usetikzlibrary{arrows}
\definecolor{uququq}{rgb}{0.25,0.25,0.25}
\definecolor{uququq}{rgb}{0.25,0.25,0.25}
\definecolor{ffqqww}{rgb}{1,0,0.4}
\definecolor{ffqqqq}{rgb}{1,0,0}
\definecolor{ttffqq}{rgb}{0.2,1,0}
\definecolor{ffffqq}{rgb}{1,1,0}
\usepackage{graphicx, tikz}
\usepackage{pgf,tikz}
\usetikzlibrary{arrows}
\usepackage[active]{srcltx}
\usepackage[latin1]{inputenc}
\usepackage{amsmath,bbm}
\usepackage{amsfonts}
\usepackage{amsthm}
\usepackage{amssymb}
\usepackage{amsfonts}
\usepackage{hyperref}
\usepackage{hyperref}
\usepackage{amsmath,accents,eqnarray}
\usepackage[english]{babel}
\usepackage{marginnote}
\usepackage[left=2.95cm,right=2.95cm,top=2.8cm,bottom=2.8cm]{geometry}
\numberwithin{equation}{section}
\usepackage{mathrsfs}
\usepackage{color}
\usepackage{amsmath, bbm}
\usepackage{amsfonts}
\usepackage{amssymb}
\usepackage{graphicx}%
\usepackage[active]{srcltx}
\usepackage[latin1]{inputenc}
\usepackage{amsmath,bbm}
\usepackage{amsfonts}
\usepackage{pgf,tikz}
\usetikzlibrary{arrows}
\usepackage{amssymb}
\newtheorem{theorem}{Theorem}[section]
\newtheorem{proposition}[theorem]{Proposition}

\newtheorem{corollary}[theorem]{Corollary}

\numberwithin{equation}{section}
\usepackage{pgf,tikz}
\usetikzlibrary{arrows}
\definecolor{fftttt}{rgb}{1,0.2,0.2}
\numberwithin{equation}{section}
\usepackage{mathrsfs}

\DeclareMathOperator{\spt}{spt}

\usepackage{pgf,tikz}
\usetikzlibrary{arrows}
\definecolor{ffqqww}{rgb}{1,0,0.4}
\definecolor{ffqqww}{rgb}{1,0,0.4}
\definecolor{xdxdff}{rgb}{0.49,0.49,1}
\begin{document}
\title[Summability of transport densities]
{Lack of regularity of the transport density\\ in the Monge problem}
\author[S. Dweik]{Samer Dweik}
\address{Laboratoire de Math\'ematiques d'Orsay, Univ. Paris-Sud, CNRS, Universit\'e Paris-Saclay, 91405 Orsay Cedex, France}
\email{samer.dweik@math.u-psud.fr}
\maketitle
\begin{abstract}
In this paper, we provide a family of counter-examples to the regularity of the transport density in the classical
Monge-Kantorovich problem. We prove that the $W^{1,p}$ regularity of the source and target measures $f^\pm$ does not imply that the transport density $\sigma$ is $W^{1,p}$, that the $BV$ regularity of $f^\pm$ does not imply that $\sigma$ is $BV$ and that $f^\pm \in C^\infty$ does not imply that $\sigma$ is $W^{1,p}$, for large $p$.
\end{abstract}
\section{Introduction}
The mass transport problem
dates back to a work from 1781 by Gaspard Monge, {\it M\'emoire sur la th\'eorie des d\'eblais et des remblais} (\cite{Monge}), where he formulated a natural question in economics which deals with the optimal way of moving points from one mass distribution to another so that the total work done is minimized. In his work, the cost of moving one unit of mass from $x$ to $y$ is measured with the Euclidean distance $|x-y|$, even though many other cost functions have been studied later on. \\ 

In order to explain this problem in full details, let us consider $f^{\pm}$ two given finite positive Borel measures in $\,\Omega\,$ satisfying the mass balance condition $f^+(\Omega)=f^-(\Omega),$ where $\Omega$ is a compact convex set in $\mathbb{R}^d$. Let $|.|$ stand for the Euclidean norm in $\mathbb{R}^d$. Then, the classical Monge optimal transportation problem (\cite{Monge}) consists in finding a transport map $T:\,\Omega \mapsto \Omega$ minimizing the functional \\
$$ T \mapsto \int_{\Omega}|x-T(x)|\,\mathrm{d}f^+,$$ \\
among all Borel measurable maps $\,T:\,\Omega \mapsto \Omega$ which satisfy the ``push-forward" condition
$T_\#f^+=f^-$, i.e 
 $$ f^-(A)= f^+(T^{-1}(A))\;\mbox{for every Borel set}\; A \subset \Omega.$$ \\
 The existence of optimal maps was addressed by many authors \cite{1},\,\cite{CafFelMcC},\,\cite{Gangbo},\,\cite{Sudakov}\,\,and\,\,\cite{wang} (see \cite{Champion} for the most general result, which is valid for arbitrary norms $||x-y||$; however in this paper we will concentrate only on the Euclidean case). Although this problem may have no solutions, its relaxed setting (which is the Kantorovich problem \cite{Kanto}) always has one. The relaxed problem consists in finding a Borel measure $\lambda$ over $\Omega \times \Omega$ (called $\textit{optimal transport plan}$) satisfying $\pi_{\#}^{\pm} \lambda = f^{\pm}$, where $\pi^{\pm}:\,\Omega\times \Omega \mapsto \Omega$ being the projections on the first and second factor, respectively (i.e $\pi^{\pm}(x^+,x^-):=x^{\pm}$), which minimizes the functional \\
 $$ \lambda \mapsto \int_{\Omega \times \Omega }|x-y|\,\mathrm{d}\lambda$$ \\
 among all Borel measures $\lambda$ on $\Omega \times \Omega\,$ satisfying $\pi_{\#}^{\pm} \lambda = f^{\pm}$. For the details about Optimal Transport theory, its history, and the main results, we refer to \cite{8} and \cite{11}.
  It is also possible to prove that the maximization of the functional
$$u \mapsto \int_{\Omega}u\,\mathrm{d}(f^+-f^-),$$\\
among all the 1-Lipschitz functions $u$ on $\Omega$,
is the dual to the Kantorovich problem (a maximizer for this problem is called {\it Kantorovich potential}). 
This duality implies that optimal $\lambda$ and $u$ satisfy $u(x)-u(y)=|x-y|$ on the support of $\lambda$, but also that, whenever we find some admissible $\lambda$ and $u$\, satisfying $\,\int_{\Omega \times \Omega}|x-y|\mathrm{d}\lambda=\int_{\Omega}u\mathrm{d}(f^+-f^-)$, they are both optimal (a maximal segment $[x,y]$ such that $u(x)-u(y)=|x-y|$ is called {\it{transport ray}}).
  In such a theory it is classical to associate with any optimal transport plan $\lambda$ a positive measure $\sigma$ on $\Omega$, called $\textit{transport density}$, which represents the amount of transport taking place in each region of $\Omega$. This measure $\sigma$ is defined by 
\begin{equation} \label{definition de la density}
<\sigma,\varphi>=\int_{\Omega \times \Omega}\mathrm{d}\lambda(x,y)\int_{0}^{1}\varphi((1-t)x+ty)|x-y|\,\mathrm{d}t\;\;\;\mbox{for all}\;\;\varphi \in C(\Omega).
\end{equation}
It is well known that $(\sigma,u)$ solves a particular PDE system, called Monge-Kantorovich system:
\begin{equation}\label{MKsyst2}
\begin{cases}
-\nabla\cdot (\sigma\nabla u)=f^+ -f^- &\mbox{ in }\Omega\\
\sigma \nabla u \cdot n=0 &\mbox{on }\partial\Omega,\\
|\nabla u|\leq 1&\mbox{ in }\Omega,\\
|\nabla u|= 1&\sigma-\mbox{a.e. }\end{cases}
\end{equation}
The $L^p$ regularity of the transport density $\sigma$ is proved successively by many authors (see, for instance, \cite{55,66,77,33,9}). In particular, we have the following

\begin{proposition}\label{prop transp dens}
 Suppose $f^+\ll\mathcal{L}^d$ or $f^-\ll\mathcal{L}^d$. Then, the transport density $\sigma$ is unique (i.e does not depend on the choice of the optimal transport plan $\lambda$) and $\sigma \ll \mathcal{L}^d$.
Moreover, if both $ f^+, f^- \in L^p(\Omega)$, then $\sigma$ also belongs to $L^p(\Omega)$.\end{proposition}

The higher order regularity of the transport density $\sigma$ is still widely open; the only known results are in $\mathbb{R}^2$\,: if $f^{\pm}$ are two positive densities, continuous and have compact, disjoint, convex support, then the ``monotone optimal transport map" $T$ is continuous except on a negligible set (the endpoints of transport rays) and the transport density $\sigma$ is actually continuous everywhere (\cite{Pratelli}). Moreover, in \cite{Wang2016}, the authors prove the continuity of the same map $T$ under the assumptions that $f^{\pm}$ are two positive densities, continuous with $\spt(f^+) \subset \spt(f^-)$ and one of the sets $\{f^+>f^-\}$, $\{f^->f^+\}$ is convex (it will be that the transport density $\sigma$ is also continuous in this case). Other results exist as far as the regularity in some directions is concerned: in \cite{Gangbo}, it has been proven that when $f^{\pm}$ are Lipschitz continuous with disjoint supports (and with some extra technical condition on the supports), then the transport density is locally Lipschitz continuous ``along transport rays". Also in \cite{Buttazo}, the authors have a more general result for the case of just summable $f^{\pm}$ without any extra conditions on supports; they prove that if $f^{\pm} \in L^p(\Omega)$, then for a.e $x \in \Omega$, the transport density $\sigma \in W_{loc}^{1,p}(R_x)$, where $R_x$ is the transport ray passing through $x$. As one can see, the $\,W^{1,p}(\Omega)\,(C^{0,\alpha}(\Omega),BV(\Omega),...)$ regularity of the transport density $\sigma$ is an interesting question, and the aim of this paper is to give a (negative) answer to it!
\\
 
  In this paper we focus on examples relating the regularity of the initial data $f^\pm$ with the regularity of the transport density $\sigma$. As  a  starting  point,  the  following  example  shows  that  in
general, the transport density $\sigma$ is not more regular than  the initial data\,: consider $\chi^+:=[0,1]^2,\,\chi^-:=[2,3] \times [0,1]$ and set $f^+(x_1,x_2):=f_1(x_1)f_2(x_2),$ where we suppose that $f^+$ is concentrated on $\chi^+$, and take $f^-(x_1,x_2):=f^+(x_1 - 2,x_2)$, for every $(x_1,x_2) \in \chi^-$. In this case, it is easy to compute the transport density $\sigma$ between $f^\pm$, so we get $\sigma(x_1,x_2)=(\int_0^{x_1} f_1(t)\,\mathrm{d}t)f_2(x_2)$ for every $x \in \chi^+$. Hence, the transport density $\sigma$ has the same regularity as $f^{\pm}$ in the $x_2$-variable. Yet, we will give examples where the regularity of the transport density $\sigma$ is worse than the regularity of the initial data $f^{\pm}$. In particular, we will prove among others the following statements:
\begin{eqnarray} \label{Results}
f^{\pm} \in W^{1,p}(\Omega)\; &\not\Rightarrow &\;\sigma \in W^{1,p}(\Omega),\;\forall\,\;p>1, \\
f^{\pm} \in BV(\Omega)\; &\not\Rightarrow & \;\sigma \in BV(\Omega),\\
f^{\pm} \in C^{\infty}(\bar{\Omega})\; &\not\Rightarrow &\;\sigma \in W^{1,3}(\Omega).
\end{eqnarray}
\section{Main Results } \label{sec.2}
Inspired by \cite{Colombo,Wang}, we will construct a family of counter-examples by, first, choosing which lines will be transport rays.
Set $\gamma>0$ and consider the following transport rays:
\begin{equation} \label{rays}
 l_a:=\bigg\{(x_1,x_2) \in \mathbb{R}^2:\,x_2=\frac{a^{\gamma}}{2} \,(x_1 + a),\,x_1 \in (-a,1) \bigg\},\;\;\;\;a \in [0,1]. 
 \end{equation} 
It is clear that the segments $l_a$ do not mutually intersect. The domain representing both source and target will be $\Delta \subset \mathbb{R}^2$ (see Figure \ref{fig1}), where  
\begin{equation} \label{domain}
 \Delta:= \mbox{interior of the triangle with vertices}\; (-1,0),\,(1,0) \; \mbox{and}\;(1,1).
\end{equation}
The initial and final density will have the form\\ 
 \begin{equation} \label{data}
 f^{+}(x_1,x_2)=1,\;f^-(x_1,x_2)= 1 + \beta(\zeta^{\prime\prime}(x_1) + \eta^{\prime\prime}(x_2))\;\;\;\mbox{for all}\;\;x:=(x_1,x_2) \in \Delta,
 \end{equation} \\
 where $\,\zeta(x_1):= -x_1^2(x_1-1)^2$ (the choice of $\zeta$ is made essentially in such a way that $\zeta(1)=\zeta^\prime(1)=0$), 
 \begin{figure}
\begin{tikzpicture}[line cap=round,line join=round,>=triangle 45,x=0.9cm,y=0.9cm]
\draw[->,color=black] (-4.04,0) -- (4.48,0);
\foreach \x in {-4,-3,-2,-1,1,2,3,4}
\draw[->,color=black] (0,-3.08) -- (0,3.42);
\foreach \y in {-3,-2,-1,1,2,3}
\clip(-4.04,-3.08) rectangle (4.48,3.42);
\draw (2,2)-- (2,0);
\draw [color=ffqqqq] (-3,0)-- (2,2);
\draw [color=ffqqqq] (0,0)-- (2,0);
\draw [color=ffqqqq] (-1,0)-- (2,0.31);
\draw [color=ffqqqq] (-1.82,0)-- (2,0.78);
\draw [color=ffqqqq] (-2.52,0)-- (2,1.44);
\begin{scriptsize}
\draw[color=black] (0.6,0.7) node {$l_a$};
\draw[color=black] (-1.9,-0.2) node {$-a$};
\end{scriptsize}
\end{tikzpicture}
\caption{ \label{fig1}}
\end{figure}
$\eta$ is a $C^2$ function with $\eta(0)=\eta^{\prime}(0)=0$ and $\beta >0$ is chosen so that $f^-$ will be a non-negative density. Note that $\eta$ is constructed in such a way that 
  the following mass balance condition for the region in the domain below each $l_a$ is satisfied:
  \begin{equation} \label{Existence of a transport map}
  \int_{\Delta_a} f^+ = \int_{\Delta_a} f^- \;\;\;\;\mbox{for all}\;a \in [0,1],
  \end{equation}
 where $\,\Delta_a$ is the subgraph of $\,l_a$ in $\,\Delta$, namely the triangle formed by $(-a,0),\,(1,0)$ and $(1,\frac{a^{\gamma}}{2}(1+a))$; 
 or equivalently, (\ref{Existence of a transport map}) can be rewritten as 
 $$
 - \int_{\Delta_a} \zeta^{\prime\prime}(x_1)\,\mathrm{d}x_1\,\mathrm{d}x_2 =  \int_{\Delta_a} \eta^{\prime\prime}(x_2)\,\mathrm{d}x_1\,\mathrm{d}x_2 \;\;\;\;\mbox{for all}\;a \in [0,1].$$\\
 Yet, it is easy to see that
 \begin{eqnarray*}
- \int_{\Delta_a} \zeta^{\prime\prime}(x_1)\,\mathrm{d}x_1\,\mathrm{d}x_2 &=&
-\int_0^{\frac{a^{\gamma}}{2}(1+a)} \int_{\frac{2}{a^{\gamma}}x_2 - a}^{1} \zeta^{\prime\prime}(x_1)\,\mathrm{d}x_1\,\mathrm{d}x_2 \\ \\
 &=& \int_0^{\frac{a^{\gamma}}{2}(1+a)} \zeta^{\prime}\bigg(\frac{2}{a^{\gamma}}x_2 - a\bigg) \,\mathrm{d}x_2 \\ \\
 &=&\frac{a^{\gamma+2}}{2}(1+a)^2
 \end{eqnarray*}
and as $\,\eta(0)=\eta^{\prime}(0)=0$, we have
\begin{eqnarray*}
\int_{\Delta_a} \eta^{\prime\prime}(x_2)\,\mathrm{d}x_1\,\mathrm{d}x_2 &=&
\int_{-a}^{1} \int_0^{\frac{a^{\gamma}}{2}(x_1+a)}\eta^{\prime\prime}(x_2)\,\mathrm{d}x_2\,\mathrm{d}x_1 \\ \\
&=&\int_{-a}^{1} \eta^{\prime}\bigg(\frac{a^{\gamma}}{2}(x_1+a)\bigg)\mathrm{d}x_1 \\ \\
&=& \frac{\eta(\frac{a^\gamma}{2}(1+a))}{\frac{a^\gamma}{2}}.
\end{eqnarray*}
 Then,
 $$ \eta(s)=s^2 a^2(s),\;\;\;\forall\;s \in (0,1)$$\\
 where $a(s)$ is the unique solution of \begin{equation} \label{implicit function a}
 s=\frac{a^{\gamma}}{2}(1+a).
 \end{equation} \\
 Yet, by the implicit function theorem, it is easy to see that there exists a $C^{\infty}$ function $h$ defined in a neighborhood of $\,0\,$ such that
 $$ s=\frac{h(s)}{2^{\frac{1}{\gamma}}}(1+h(s))^{\frac{1}{\gamma}}.$$\\
 Hence, $h(s^{\frac{1}{\gamma}})$ is a solution to (\ref{implicit function a}) and $\eta(s)=s^2 h^2(s^{\frac{1}{\gamma}})$. After tedious computations, we can check that
 
 $$\eta^{\prime}(s)=
 2s\,h^2 (s^{\frac{1}{\gamma}})\,+\,\frac{2}{\gamma}s^{\frac{1}{\gamma}+1}\,\,h(s^{\frac{1}{\gamma}})\,h^{\prime}(s^{\frac{1}{\gamma}}),$$
 
 $$\eta^{\prime\prime}(s)=
 2\,h^2 (s^{\frac{1}{\gamma}})\,+\,\bigg(\frac{6}{\gamma}\, +\, \frac{2}{\gamma^2}\bigg)s^{\frac{1}{\gamma}}\,\,h(s^{\frac{1}{\gamma}})\,h^{\prime}(s^{\frac{1}{\gamma}})\,+\, \frac{2}{\gamma^2}\,s^{\frac{2}{\gamma}}(\,(h^{\prime}(s^{\frac{1}{\gamma}}))^2 \,+\, \,h(s^{\frac{1}{\gamma}})\,h^{\prime\prime}(s^{\frac{1}{\gamma}})),$$
 
$$ \eta^{\prime\prime\prime}(s)=
 \bigg(\frac{4}{\gamma}\,+\,\frac{6}{\gamma^2}\, +\, \frac{2}{\gamma^3}\bigg)s^{\frac{1}{\gamma}-1}\,\,h(s^{\frac{1}{\gamma}})\,h^{\prime}(s^{\frac{1}{\gamma}})\,+\, \bigg(\frac{6}{\gamma^3}+\,\frac{6}{\gamma^2}\bigg)\,s^{\frac{2}{\gamma}-1}(\,(h^{\prime}(s^{\frac{1}{\gamma}}))^2 \,+\, \,h(s^{\frac{1}{\gamma}})\,h^{\prime\prime}(s^{\frac{1}{\gamma}}))$$
 
 $$+\,\frac{6}{\gamma^3}\,s^{\frac{3}{\gamma}-1}\,\,h^\prime(s^{\frac{1}{\gamma}})\,h^{\prime\prime}(s^{\frac{1}{\gamma}})+\,\frac{2}{\gamma^3}\,s^{\frac{3}{\gamma}-1}\,\,h(s^{\frac{1}{\gamma}})\,h^{\prime\prime\prime}(s^{\frac{1}{\gamma}}),$$
 and\\
 $$\eta^{\prime\prime\prime\prime}(s)=
 \bigg(\frac{14}{\gamma^4}\,+\,\frac{12}{\gamma^3}\, -\, \frac{2}{\gamma^2}\bigg)s^{\frac{2}{\gamma}-2}\,(h^{\prime}(s^{\frac{1}{\gamma}}))^2+\,
 \bigg(\frac{2}{\gamma^4}\,+\,\frac{4}{\gamma^3}\, -\, \frac{2}{\gamma^2}-\frac{4}{\gamma}\bigg)s^{\frac{1}{\gamma}-2}\,\,h(s^{\frac{1}{\gamma}})\,h^{\prime}(s^{\frac{1}{\gamma}})\,+\, \frac{6}{\gamma^4}\,s^{\frac{4}{\gamma}-2}(h^{\prime\prime}(s^{\frac{1}{\gamma}}))^2 $$
 
 $$+\,\bigg(\frac{14}{\gamma^4}\,+\,\frac{12}{\gamma^3}\, -\, \frac{2}{\gamma^2}\bigg)s^{\frac{2}{\gamma}-2}\,\,h(s^{\frac{1}{\gamma}})\,h^{\prime\prime}(s^{\frac{1}{\gamma}})\,+\,\bigg(\frac{30}{\gamma^4}\,+\,\frac{12}{\gamma^3}\bigg)s^{\frac{3}{\gamma}-2}\,\,h^\prime(s^{\frac{1}{\gamma}})\,h^{\prime\prime}(s^{\frac{1}{\gamma}})\,+\,
 \frac{8}{\gamma^4}\,s^{\frac{4}{\gamma}-2}\,h^\prime(s^{\frac{1}{\gamma}})\,h^{\prime\prime\prime}(s^{\frac{1}{\gamma}})
 $$
 $$+\,\bigg(\frac{12}{\gamma^4}\,+\,\frac{4}{\gamma^3}\bigg)s^{\frac{3}{\gamma}-2}\,\,h(s^{\frac{1}{\gamma}})\,h^{\prime\prime\prime}(s^{\frac{1}{\gamma}})\,+\,
 \frac{2}{\gamma^4}\,s^{\frac{4}{\gamma}-2}\,h(s^{\frac{1}{\gamma}})\,h^{\prime\prime\prime\prime}(s^{\frac{1}{\gamma}}).
 $$
Hence,
 \begin{equation}\label{The regularity of the initial data}
\begin{cases}
\eta(|.|) \in C^{\infty}(\mathbb{R})   & \;\;\;\mbox{if}\;\;\gamma=\frac{1}{2},\\
\eta \in C^{\infty}[0,+\infty)\;\,\mbox{and}\,\;\eta(|.|)\in C^{4,1}(\mathbb{R})  &\;\;\; \mbox{if}\;\;\gamma=1,\\
\eta(|.|) \in C^{3,\frac{2}{\gamma}-1}(\mathbb{R})  & \;\;\;\mbox{if}\;\;1<\gamma < 2,\\
\eta \in C^{3}[0,+\infty) \;\,\mbox{and}\,\;\eta(|.|)\in C^{2,1}(\mathbb{R}) & \;\;\;\mbox{if}\;\;\gamma =2,\\
\eta(|.|) \in C^{2,\frac{2}{\gamma}}(\mathbb{R}) \cap W^{3,\frac{\gamma}{\gamma-2}-\varepsilon}(\mathbb{R}) &\;\;\; \mbox{if}\;\;\gamma > 2,\,\varepsilon>0.
\end{cases}
\end{equation}\\
Now, we will introduce the following key propositions, whose proofs, for simplicity of exposition, are postponed to Section \ref{Sec.3}.
 \begin{proposition}\label{Prop2.1}
The transport density $\sigma$ between $f^+$ and $\,f^-$ is not in $W^{1,p}(\Delta)$ for all $p$ satisfying
$$ p  \geq \min\bigg\{\frac{\gamma}{\gamma-1},\frac{\gamma+2}{\gamma}\bigg\}.$$ 
 \end{proposition}
 \begin{corollary}\label{Statements Sobolev}
 We have the following statements:
\begin{eqnarray}
 f^{\pm} \in W^{1,p}(\Delta)\; &\not\Rightarrow &\;\sigma \in W^{1, \frac{2p+\varepsilon}{p+1}}(\Delta),\;\forall\;\varepsilon>0,\;p>1, \\
 f^{\pm} \in C^{1}(\bar{\Delta})\; &\not\Rightarrow &\;\sigma \in H^1(\Delta), \\
f^{\pm} \in C^{1,\alpha}(\Delta)\; &\not\Rightarrow &\;\sigma \in W^{1,2+\alpha}(\Delta),\;\forall\;\alpha \in (0,1), \\
f^{\pm} \in C^{\infty}(\bar{\Delta})\; &\not\Rightarrow &\;\sigma \in W^{1,3}(\Delta).
\end{eqnarray}
\end{corollary}
\begin{proof}
These statements follow immediately from (\ref{The regularity of the initial data}) and the proposition \ref{Prop2.1}. Indeed, for $\gamma>2$: $\eta^{\prime\prime} \in W^{1,\frac{\gamma}{\gamma-2}-\varepsilon}(\mathbb{R})$ ($\mbox{for all}\,\,\varepsilon>0$) and the transport density $\sigma$ is not in $W^{1,\frac{\gamma}{\gamma-1}}(\Delta)$, so (2.7) follows. To prove (2.8), take $\gamma=2$ and then, in this case, we have that $\eta^{\prime\prime} \in C^{1}(\mathbb{R})$ and $\sigma \notin H^{1}(\Delta)$. For $1<\gamma<2$: $\eta^{\prime\prime} \in C^{1,\frac{2}{\gamma}-1}(\mathbb{R})$ and $\sigma \notin W^{1,\frac{\gamma+2}{\gamma}}(\Delta)$, so (2.9) follows. Finally, for $\gamma=1$: $\eta^{\prime\prime} \in C^{\infty}(\mathbb{R})$ and $\sigma \notin W^{1,3}(\Delta)$ and the statement (2.10) follows.  $\qedhere$
\end{proof} 
\begin{proposition}\label{Holderregularity}
The transport density $\sigma$ between $f^+$ and $\,f^-$ is not in $C^{\,0,\frac{1}{\gamma+1}+\varepsilon}(\Delta)$, for all $\,\varepsilon>0$.
 \end{proposition}
 \begin{corollary}\label{Statements Sobolev}
 We have the following statements:
\begin{eqnarray}
f^{\pm} \in C^{0,\alpha}(\Delta)\; &\not\Rightarrow &\;\sigma \in C^{0,\frac{\alpha}{\alpha+2}+\varepsilon}(\Delta),\;\forall\;\varepsilon>0,\,\alpha  \in (0,1),\\
f^{\pm} \in C^{1}(\bar{\Delta})\; &\not\Rightarrow &\;\sigma \in C^{0,\frac{1}{3}+\varepsilon}(\Delta),\;\forall\;\varepsilon>0,\\
f^{\pm} \in C^{1,\alpha}(\Delta)\; &\not\Rightarrow &\;\sigma \in C^{0,\frac{1+\alpha}{3+\alpha}+\varepsilon}(\Delta),\;\forall\;\varepsilon>0,\;\alpha \in (0,1) ,\\
f^{\pm} \in C^{\infty}(\bar{\Delta})\; &\not\Rightarrow & \;\sigma \in C^{0,\frac{1}{2} + \varepsilon}(\Delta),\;\forall\;\varepsilon>0.
\end{eqnarray}
\end{corollary}
\begin{proof}
These statements follow immediately from (\ref{The regularity of the initial data}) and the proposition \ref{Holderregularity}. Indeed, for $\gamma>2$: $\eta^{\prime\prime} \in C^{0,\frac{2}{\gamma}}(\mathbb{R})$, for $\gamma=2$: $\eta^{\prime\prime} \in C^{1}(\mathbb{R})$, for $1<\gamma<2$: $\eta^{\prime\prime} \in C^{1,\frac{2}{\gamma}-1}(\mathbb{R})$ and, finally, for $\gamma=1$: $\eta^{\prime\prime} \in C^{\infty}(\mathbb{R})$. Yet, in all these cases, the transport density $\sigma \notin C^{0,\frac{1}{\gamma}+\varepsilon}(\Delta)$, for all $\varepsilon>0$.  $\qedhere$
\end{proof} 
To obtain counter-examples to interior regularity of the transport density, it suffices to reflect the domain across the $x_1$-axis. Let $\Delta^{\prime}$ be the reflection of $\Delta$ with respect to the $x_1$-axis (see Figure \ref{Fig.2}) and set $\Omega:=\Delta \cup \Delta^{\prime}$.
\begin{figure} \label{Fig.2}
\begin{tikzpicture}[line cap=round,line join=round,>=triangle 45,x=0.9cm,y=0.9cm];
\clip(-4.22,-2.93) rectangle (5.14,4.27);
\draw (2,2)-- (2,0);
\draw [color=fftttt] (2,2)-- (-3.18,0);
\draw [color=fftttt] (0,0)-- (2,0);
\draw (0,0)-- (-3.18,0);
\draw [color=fftttt] (-2.66,0)-- (2,1.4);
\draw [color=fftttt] (-1.42,0)-- (2,0.6);
\draw [color=fftttt] (-0.8,0)-- (2,0.22);
\draw (2,-2)-- (-3.18,0);
\draw (2,-2)-- (2,0);
\draw (-2.66,0)-- (2,-1.4);
\draw (-1.42,0)-- (2,-0.6);
\draw (-0.8,0)-- (2,-0.22);
\begin{scriptsize}
\draw[color=fftttt] (1,0.7) node {$\Delta$};
\draw[color=black] (1,-0.7) node {$\Delta^{\prime}$};
\end{scriptsize}
\end{tikzpicture}
\caption{ \label{fig2}}
\end{figure}
Extend the functions $f^{\pm}$ to $\Omega$ so that they are symmetric with respect to the $x_1$-axis. Let $T$ be an optimal transport map between $f^\pm$ and let $\sigma$ be the transport density between them, then it is easy to prove that the map $S$, which is equal to $T$ on $\Delta$ and to the reflection of $T$ with respect to the $x_1$-axis on $\Delta^{\prime}$, is an optimal transport map between the extended densities and the transport density between them is equal to $\sigma$ on $\Delta$ and to the reflection of $\sigma$, with respect to the $x_1$-axis, on $\Delta^{\prime}$.
Using this fact and (\ref{The regularity of the initial data}), we get the following statements:
\begin{eqnarray}
 f^{\pm} \in W^{1,p}(\Omega)\; &\not\Rightarrow &\;\sigma \in W^{1, \frac{2p+\varepsilon}{p+1}}_{loc}(\Omega),\;\forall\;\varepsilon>0,\;p>1, \\
 f^{\pm} \in C^{0,\alpha}(\Omega)\; &\not\Rightarrow &\;\sigma \in C_{loc}^{0,\frac{\alpha}{\alpha+2}+\varepsilon}(\Omega),\;\forall\;\varepsilon>0,\;\alpha  \in (0,1),\\
 f^{\pm} \in C^{0,1}(\Omega)\; &\not\Rightarrow &\;\sigma \in H^1_{loc}(\Omega) \cup C^{0,\frac{1}{3}+\varepsilon}_{loc}(\Omega),\;\forall\;\varepsilon>0,\\
f^{\pm} \in C^{1,\alpha}(\Omega)\; &\not\Rightarrow &\;\sigma \in W^{1,2+\alpha}_{loc}(\Omega) \cup C_{loc}^{0,\frac{1+\alpha}{3+\alpha}+\varepsilon}(\Omega),\;\forall\;\varepsilon>0,\;\alpha \in (0,1), \\
f^{\pm} \in C^{2,1}(\Omega)\; &\not\Rightarrow &\;\sigma \in W_{loc}^{1,3}(\Omega) \cup C_{loc}^{0,\frac{1}{2} + \varepsilon}(\Omega),\;\forall\;\varepsilon>0, \\
f^{\pm} \in C^{\infty}(\bar{\Omega})\; &\not\Rightarrow & \;\sigma \in W_{loc}^{1,5}(\Omega) \cup C_{loc}^{0,\frac{2}{3} + \varepsilon}(\Omega),\;\forall\;\varepsilon>0.
\end{eqnarray}
 \section{Proof} \label{Sec.3}
 In this section, we want to prove Propositions \ref{Prop2.1} \& \ref{Holderregularity}. Firstly, we will compute the transport density $\sigma$ between $f^+$ and $f^-$. To do that, let us observe that the family $\{l_a,\, a \in (0,1)\}$, where $l_a$ is defined as in (\ref{rays}), covers $\,\Delta\,$ so that for every $ x:=(x_1,x_2) \in \Delta$, there exists a unique pair $(t,a) \in (0,1)^2$ such that $x \in l_a$ and $|x-(-a,0)|=t\, l(a)$, where $l(a)$ is the length of $l_a$. In other words, we have
\begin{eqnarray*}
x&=&\left(-a+(1+a)t\,,\,(1+a)\frac{t a^{\gamma}}{2}\right).
\end{eqnarray*}
Fix $(t,a)\in (0,1)^2\,$ and set, \\
$$\omega_{\varepsilon}:=\bigg\{\bigg(-s+(1+s)\tau,\,(1+s)\frac{\tau s^{\gamma}}{2}\bigg),\,(\tau,s) \in [0,t] \times [a,a+\varepsilon]\bigg\}$$  
\\
where $\varepsilon > 0$ is small enough. Recalling (\ref{MKsyst2}) and integrating $ -\nabla \cdot (\sigma \nabla u ) = f\,$ on $\,\omega_{\varepsilon}$, we get \\ \begin{equation} \label{divergence formula 1}
-\int_{\partial \omega_{\varepsilon}} \sigma \nabla u\cdot n = \int_{\omega_{\varepsilon}} f.
\end{equation}\\
Suppose that the family of segments $(l_a)_{a\in(0,1)}$ are, in fact, all the transport rays on which the optimal transport map, between $f^+$ and $f^-$, acts. In this case, we get that for every $x \in l_a$:\\
$$ \nabla u (x)= \frac{(-a,0)-(1,(1+a)\frac{a^{\gamma}}{2})}{|(-a,0)-(1,(1+a)\frac{a^{\gamma}}{2})|}=\frac{-(1,\frac{a^{\gamma}}{2})}{\sqrt{1 + (\frac{a^{\gamma}}{2})^2}},$$ \\ 
which means that $\nabla u (x) \cdot n=0\,$ if $\,n$ is the unit orthogonal vector to $\,l_a$. 
Hence, (\ref{divergence formula 1}) becomes \\
\begin{equation} \label{divergence formula 2}
-\int_{s_{\varepsilon}} \sigma \nabla u\cdot n =\int_{\omega_{\varepsilon}} f
\end{equation}\\
where $\,s_{\varepsilon}:=\bigg\{(-s+(1+s)t,\,(1+s)\frac{t s^{\gamma}}{2}),\,s \in [a,a+\varepsilon]\bigg\}$. Yet, we have \\
\begin{eqnarray*}
\int_{\omega_{\varepsilon}} f(x_1,x_2) \,\mathrm{d}x_1\, \mathrm{d} x_2 &=& \int_{\omega_{\varepsilon}} -\beta(\zeta^{\prime\prime}(x_1) + \eta^{\prime\prime}(x_2))  \,\mathrm{d}x_1 \,\mathrm{d} x_2 \\ \\
&=& \int_a^{a+\varepsilon}\int_0^t -\beta\bigg(\zeta^{\prime\prime}(-s+(1+s)\tau) + \eta^{\prime\prime}\bigg((1+s)\frac{\tau s^{\gamma}}{2}\bigg)\bigg) \, J(\tau,s) \, \mathrm{d}\tau\, \mathrm{d}s,
\end{eqnarray*} \\
where $\,J:=|\det(D_{(\tau,s)}(x_1,x_2))|$. Yet,  \\ 

$$D_{(t,a)}(x_1,x_2) :=\begin{pmatrix}
 \partial_t x_1 &&  \partial_a x_1 \\
 \\ 
 \partial_t x_2 &&  \partial_a x_2
 \end{pmatrix}=\begin{pmatrix}
 1+a &&  -1+t \\
 \\ 
 \frac{(1+a)\,a^{\gamma}}{2} && (\gamma\,(1+a) + a)\,\frac{t a^{\gamma-1}}{2}
 \end{pmatrix}.$$\\ 
Then, \begin{equation} \label{Jacobien}
J(t,a)=(1+a)(\gamma \,(1+a)\, t  + a)\frac{\,a^{\gamma-1}}{2}.
\end{equation}\\
On the other hand, 
$$-\nabla u \cdot n =\frac{\partial_t x}{|\partial_t x|}\cdot R\,\partial_a x=\frac{J(t,a)}{l(a)}$$\\
where $R:=\begin{pmatrix}
 0 & 1 \\
 -1 & 0
 \end{pmatrix}$ is the rotation matrix. Hence,\\
\begin{eqnarray*}
-\int_{s_{\varepsilon}} \sigma \nabla u.n &=& \int_a^{a+\varepsilon} \sigma\bigg(-s+(1+s)t,\,(1+s)\frac{t s^{\gamma}}{2}\bigg) \frac{J(t,s)}{l(s)}\,\mathrm{d}s. 
\end{eqnarray*}\\ 
By (\ref{divergence formula 2}), we infer that\\
\begin{multline*} 
\lim_{\varepsilon \to 0^+}\, \frac{1}{\varepsilon} \int_a^{a+\varepsilon} \sigma\bigg(-s+(1+s)t,\,(1+s)\frac{t s^{\gamma}}{2}\bigg) \frac{J(t,s)}{l(s)}\,\mathrm{d}s \\ \\
=\lim_{\varepsilon \to 0^+}\, \frac{1}{\varepsilon} \int_a^{a+\varepsilon}\int_0^t -\beta\bigg(\zeta^{\prime\prime}(-s+(1+s)\tau) + \eta^{\prime\prime}\bigg((1+s)\frac{\tau s^{\gamma}}{2}\bigg)\bigg) \, J(\tau,s) \, \mathrm{d}\tau\, \mathrm{d}s.
\end{multline*}
Finally, we get\\
\begin{equation} \label{transport density}
\sigma(x) =\frac{l(a)\int_0^t -\beta(\zeta^{\prime\prime}(-a + (1+a)\tau) + \eta^{\prime\prime}((1+a)\frac{\tau a^{\gamma}}{2})) \, J(\tau,a) \, \mathrm{d}\tau}{J(t,a)}.
\end{equation}\\
\\
Now, we are ready to prove Proposition \ref{Holderregularity}. Indeed, for every $\varepsilon>0$, let $(t_{\varepsilon},a_{\varepsilon})$ be in $(0,1)^2$ such that  $$x_{\varepsilon}:=(0,\varepsilon)=\bigg(-a_{\varepsilon} + (1+a_{\varepsilon})t_{\varepsilon},\,(1 +a_{\varepsilon})\frac{t_{\varepsilon} a_{\varepsilon}^{\gamma}}{2}\bigg).$$\\
As $\zeta^{\prime\prime}(0)<0\,$ and $\,\eta^{\prime\prime}(0)=0$, then, from (\ref{transport density}), we can see easily that, close to the origin, we have \\$$ \sigma\, \approx 
 \frac{\int_0^{t}  J(\tau,a) \, \mathrm{d}\tau}{J(t,a)}$$\\
 where $f \approx g\,$ means that there exist $c^\pm>0$ such that $c^- g \leq f \leq c^+ g$. Yet, by (\ref{Jacobien}), one has\\
$$\int_0^t  J(\tau,a) \, \mathrm{d}\tau= \frac{t}{2}\bigg(J(t,a) + (1+a)\frac{a^{\gamma}}{2}\bigg).$$\\ \\
As $\,t_{\varepsilon}=\frac{a_{\varepsilon}}{1+a_{\varepsilon}}$ and $\varepsilon=\frac{a_{\varepsilon}^{\gamma+1}}{2}$, we infer that $\,t_{\varepsilon} \approx \varepsilon^{\frac{1}{\gamma+1}}$. Hence,\\
$$  \sigma(x_\varepsilon)\,\approx\,\varepsilon^{\frac{1}{\gamma + 1}}.$$
\\
This completes the proof of the proposition \ref{Holderregularity}. Next, to prove Proposition \ref{Prop2.1}, we will only look at
 $\partial_{x_2} \sigma$ close to the origin and to do that, we want to compute, firstly, $\partial_t \sigma$ and $\partial_a \sigma$.
 Differentiating (\ref{transport density}) with respect to $t$ and $a$ respectively, we get\\
 \begin{equation} \label{dt density}
 \partial_t \sigma(x) = l(a) \,f(x) - \frac{\partial_t J(t,a)}{J(t,a)}\,\sigma(x),
 \end{equation}
 and 
 \begin{equation} \label{da density}
 \partial_a \sigma(x) =\frac{\partial_a l(a)}{l(a)}\,\sigma(x) \,-\,  \frac{\partial_a J(t,a)}{J(t,a)}\,\sigma(x) \,+ K_1(t,a)+K_2(t,a) \\\\
 \end{equation}\\
 where,
 $$ K_1(t,a):=\frac{l(a)\int_0^t -\beta(\zeta^{\prime\prime}(-a + (1+a)\tau) + \eta^{\prime\prime}((1+a)\frac{\tau a^{\gamma}}{2}))\,\partial_a J(\tau,a) \, \mathrm{d}\tau}{J(t,a)}$$
 and\\
 $$K_2(t,a):=\frac{l(a) \int_0^t -\beta(-(1-\tau)\zeta^{\prime\prime\prime}(-a + (1+a)\tau) + \frac{\tau a^{\gamma - 1}}{2 }(\gamma(1+a) + a)\,\eta^{\prime\prime\prime}((1+a)\frac{\tau a^{\gamma}}{2})) \, J(\tau,a) \, \mathrm{d}\tau}{J(t,a)}.$$ \\  
 \\
Now, we claim that, for any $\gamma \geq \frac{1}{2},$ 
\begin{equation} \label{upper/lower bound}
\partial_{a}\sigma \approx \,t \,+\, \bigg(\frac{t}{t+a}\bigg)^2.
\end{equation}
 From Section \ref{sec.2}, we have 
\begin{eqnarray} \label{Holder for tiers eta} 
|\eta^{\prime\prime\prime}(x_2)| \leq C\,x_2^{\frac{2}{\gamma}-1},\,\mbox{for all}\;\,x_2 \in (0,1)
\end{eqnarray}
\\
where $C:=C(\gamma)$. Then,
$$\bigg|\frac{\tau a^{\gamma - 1}}{2 }(\gamma(1+a) + a)\,\eta^{\prime\prime\prime}\bigg((1+a)\frac{\tau a^{\gamma}}{2}\bigg)\bigg| \,\leq C\,\tau^{\frac{2}{\gamma}} a.$$\\
As $\,\zeta^{\prime\prime\prime}(0)>0$, we infer that
$$ K_2(t,a) \approx t.$$\\
In addition, it is easy to see that \\
$$\frac{\partial_a l(a)}{l(a)}\,\sigma(x)=\frac{4+a^{2 \gamma} + \gamma(1 +a)\,a^{2 \gamma - 1}}{(1+a)(4+a^{2 \gamma})}\,\sigma(x) \approx t.$$
On the other hand, \\
$$ K_1(t,a)\,-\,\frac{\partial_a J(t,a)}{J(t,a)}\,\sigma(x) 
 = \frac{l(a)\int_0^t f(-a + (1+a)\tau,(1+a)\frac{\tau a^{\gamma}}{2})\,(J(t,a)\partial_a J(\tau,a)-J(\tau,a)\partial_a J(t,a)) \, \mathrm{d}\tau}{J(t,a)^2}.$$\\ \\
Yet, by (\ref{Jacobien}), we have 
 \begin{equation*}\label{2}
 \partial_a J(t,a)=\left(\gamma\,(\gamma -1)\,(1+a)^2\, t + (\gamma\,(1+a)\,(1+2 t) +a)\,a\right)\frac{a^{\gamma - 2}}{2}.
 \end{equation*} \\ Then,
it is not difficult to check that\\ 
$$J(t,a)\,\partial_{a} J(\tau,a) -  J(\tau,a)\,\partial_{a} J(t,a)=\frac{\gamma}{4}(1+a)^2 (t-\tau) a^{2 \gamma - 2}.
$$\\
Using (\ref{Jacobien}) again, we infer that
$$  K_1(t,a)\,-\,\frac{\partial_a J(t,a)}{J(t,a)}\,\sigma(x) \approx \bigg(\frac{t}{t+a}\bigg)^2.$$\\ \\
This completes the proof of (\ref{upper/lower bound}). 
Now, our aim is to prove that

\begin{equation} \label{approx}
\partial_t \sigma \approx 1.
\end{equation} \\
Fix $\delta>0$. As $\,\eta^{\prime\prime}(0)=0$, then, near the origin, we can assume that\\
 $$ |f(x) -2 \beta|< \delta.$$
 From (\ref{dt density}), we get
 \begin{eqnarray*}
 \partial_t \sigma(x) &\geq & l(a)\bigg(2 \beta - \delta - \frac{\int_0^t \partial_t J(t,a) J(\tau,a)\mathrm{d}\tau}{J(t,a)^2}(2 \beta+\delta)\bigg).
 \end{eqnarray*}\\
  Yet,
 \begin{equation*}\label{1}
 \partial_t J(t,a)=\frac{\gamma}{2}\,(1+a)^2\,a^{\gamma - 1}
 \end{equation*}\\
 and then,
 $$ \frac{\int_0^t \partial_t J(t,a)J(\tau,a)\mathrm{d}\tau}{J(t,a)^2}=\frac{1}{2}\bigg(1-\frac{a^2}{(\gamma \,(1+a)\,t  + a)^2}\bigg).$$
Hence,

 $$\partial_t \sigma(x) \geq \;\beta - \frac{3 \delta}{2}>0.$$\\ \\
In the same way, we can prove that $\partial_t\sigma$ is bounded from above and then, (\ref{approx}) follows.
Yet,  \\ 
$$ \partial_{x_2} \sigma:=\partial_t\sigma\,\partial_{x_2} t\, +\, \partial_a\sigma\,\partial_{x_2}a$$
and \\
 $$ D_{(x_1,x_2)}(t,a) :=\begin{pmatrix}
 \partial_{x_1}t  &&  \partial_{x_2} t \\
 \\ 
 \partial_{x_1}a  &&  \partial_{x_2}a
 \end{pmatrix}=
 \frac{1}{J(t,a)}\,\begin{pmatrix}
 (\gamma\,(1+a) + a)\,\frac{t a^{\gamma-1}}{2}  &&  1-t \;\\
 \\ 
 -\frac{(1+a)\,a^{\gamma}}{2} && 1+a \;
 \end{pmatrix}.$$\\
 Hence, by (\ref{upper/lower bound}) \& (\ref{approx}), we get
 $$ \partial_{x_2} \sigma \,\approx\,\frac{1}{J}$$
 and
 \begin{eqnarray*}
 || \partial_{x_2} \sigma ||_{L^p(\Delta)}^p & \approx & \int_{\Delta} \frac{1}{J(t,a)^{p}}\,\mathrm{d}x_1\,\mathrm{d}x_2 \approx  \int_0^{\delta} \int_0^{\delta} \frac{1}{J(t,a)^{p-1}}\,\mathrm{d}t\,\mathrm{d}a
 \end{eqnarray*}\\ 
 where $\delta>0$ small enough. From (\ref{Jacobien}), we get\\
 \begin{eqnarray*}
 || \partial_{x_2} \sigma ||_{L^p(\Delta)}^p & \approx & \int_0^{\delta} \int_0^{\delta} \frac{1}{a^{(\gamma-1)(p-1)}(t+a)^{p-1}} \,\mathrm{d}t\,\mathrm{d}a \\\\  \\
 & \approx & 
     \int_0^{\delta} r\,\mathrm{d}r \int_0^{\frac{\pi}{2}} \frac{1}{r^{\gamma(p-1)} \sin(\theta)^{(\gamma-1)(p-1)}(\cos(\theta)+\sin(\theta))^{(p-1)}} \,\mathrm{d}\theta \\ \\ \\
     & \approx & 
     \int_0^{\delta} \frac{1}{r^{\gamma(p-1)-1}} \,\mathrm{d}r \int_0^{\frac{\pi}{2}} \frac{1}{\sin(\theta)^{(\gamma-1)(p-1)}} \,\mathrm{d}\theta.
\end{eqnarray*}  \\ \\
Then, the proposition \ref{Prop2.1} is proved. As in \cite{Colombo,Wang}, it is not difficult to prove the existence of a Kantorovich potential $u$ to assert that the rays $(l_a)_{a\in (0,1)}$ are, in fact, all the transport rays between $f^+$ and $f^-$. This follows immediately from the fact that the unit vector of any transport ray $l_a$ is 
an irrotational vector field, which implies that there is a 1-Lipschitz function $u$ such that 
\begin{equation} \label{u of slope 1}
u(x) - u(y)=|x-y|\;\;\;\;\forall\;x,\,y \in l_a.
\end{equation}
In addition, by \cite{CafFelMcC,wang}, one can show that there is a unique measure preserving map $\,T\,$ from $(f^+,\Delta)$ to $(f^-,\Delta)$ such that $x$ and $T(x)$ lie in a common $l_a$, for all $x \in \Delta$. 
Now, it is classical to infer that $T$ is an optimal transport map between $f^\pm$ and $u$ is the corresponding Kantorovich potential.
\section{BV counter-example} \label{Sec.4}
In this section, we will prove the statement (1.4). This means that we want to construct two densities $f^\pm \in BV(\Omega)$ such that the transport density $\sigma$ between them is not in $BV(\Omega)$. First of all, we can see easily that for any $\gamma>0$, the densities $f^\pm$, which are constructed in Section \ref{sec.2}, are in $BV(\Omega)$, but it will be also the same for the transport density $\sigma$ between them. Indeed, to get a counter-example to the $W^{1,p}$ regularity of the transport density, for $p \rightarrow 1$, we need a $\gamma \rightarrow \infty$.
 Hence, to get a $BV$ counter-example, we could collect an infinity of triangles (constructed as in Section \ref{sec.2}) with a sequence of exponents $\gamma_n\to\infty$ (where $\gamma_n$ is the exponent of the slopes of the transport rays in the $n$-th triangle, see \ref{rays}). Actually, if we play on other parameters, we just need to take $\gamma_n=\gamma>1$. To do that, let us define $\Delta_n$ as follows\,:
$$\Delta_n:= \mbox{triangle with vertices}\; (-l_n,0),\,\bigg(1,-\frac{l_n^{\gamma}}{2}(1+l_n)\bigg) \;\,\mbox{and}\;\,\bigg(1,\frac{l_n^{\gamma}}{2}(1+l_n)\bigg)$$\\
where $\,l_n:=\frac{1}{n}$. Set, $\Delta_1^{\prime}:=\Delta_1$ and, for all $n \geq 2$, define $\Delta_n^{\prime}$ as a suitable roto-translation of $\Delta_n$:

$$\Delta_n^{\prime}:=\{(x_1,x_2) \in \mathbb{R}^2\,:\,(\cos(\theta_{n})(x_1+l_1) \,+\, \sin(\theta_{n}) x_2 - l_{n},-\sin(\theta_{n})(x_1+l_1) + \cos(\theta_{n}) x_2) \in \Delta_n\},$$\\
where  
$$\theta_{n}:= \sum_{k=1}^{n-1} \alpha_k + \alpha_{k+1}$$
and
$$\sin(\alpha_k):=\frac{\frac{l_k^\gamma}{2}}{\sqrt{1+\bigg(\frac{l_k^\gamma}{2}\bigg)^2}},\;\,\alpha_k \in \bigg(0,\frac{\pi}{2}\bigg).$$\\
  Finally, set
$$ \Omega:=\bigcup\limits_{n=1}^{\infty} \Delta_n^{\prime}.$$\\
\begin{figure}
\begin{tikzpicture}[line cap=round,line join=round,>=triangle 45,x=0.6cm,y=0.6cm]
\clip(-7.99,-4.33) rectangle (8.99,8.73);
\draw (4,4)-- (-4,0);
\draw (-4,0)-- (0.94,6.59);
\draw (0.94,6.59)-- (3.32,3.66);
\draw [color=ffqqww] (-3.02,0.79)-- (1.72,5.63);
\draw [color=ffqqww] (-3.02,0.79)-- (2.73,4.38);
\draw (-4,0)-- (-1.37,7.3);
\draw (-1.37,7.3)-- (0.63,6.18);
\draw (-4,0)-- (-2.9,2.03);
\draw [color=ffqqww] (-2.9,2.03)-- (-0.33,6.72);
\draw [color=ffqqww] (-3.28,1.32)-- (-0.76,6.95);
\draw [color=ffqqww] (-3.28,1.32)-- (0.07,6.49);
\draw (-4,0)-- (-1.49,2.03);
\draw [color=ffqqww] (-1.49,2.03)-- (2.22,5.01);
\draw [dash pattern=on 4pt off 4pt] (-1.48,7.01)-- (-2.34,7.24);
\draw [dash pattern=on 4pt off 4pt] (-2.34,7.24)-- (-4,0);
\draw [dash pattern=on 4pt off 4pt] (-2.38,7.06)-- (-2.81,7.12);
\draw [dash pattern=on 4pt off 4pt] (-2.81,7.12)-- (-4,0);
\draw (4,4)-- (4,-4);
\draw (4,-4)-- (-4,0);
\draw [color=ffqqww] (0,0)-- (4,0);
\draw [color=ffqqww] (-1.14,0)-- (4,0.94);
\draw [color=ffqqww] (-2.51,0)-- (4,2.29);
\draw [color=ffqqww] (-1.14,0)-- (4,-0.94);
\draw [color=ffqqww] (-2.51,0)-- (4,-2.29);
\draw (-4,0)-- (0,0);
\begin{scriptsize}
\draw[color=black] (-4.6,-0.4) node {$(-1,0)$};
\draw[color=black] (4,6) node {$\Omega$};
\draw[color=black] (3,1.4) node {$\Delta_1^{\prime}$};
\draw[color=black] (1.4,4.8) node {$\Delta_2^{\prime}$};
\draw[color=black] (-2.1,6.4) node {$\Delta_n^{\prime}$};
\end{scriptsize}
\end{tikzpicture}
\caption{ \label{fig3}}
\end{figure}\\
Fix $n \in \mathbb{N}^*$. Then, after a suitable roto-translation of axis, we can assume that $\Delta_n^{\prime}=\Delta_n$. Set,

 $$f^+(x_1,x_2):=1$$
 and
$$f^-(x_1,x_2):=f_n^-(x_1,x_2):=1\,+\,\beta(\zeta^{\prime\prime}(x_1) \,+\, \eta^{\prime\prime}(|x_2|)),\;\;\,\mbox{for all} \,\;(x_1,x_2) \in \Delta_n^{\prime}$$\\
where $\zeta$ and $\eta$ are the same functions which are constructed in the section \ref{sec.2}.  
Let us denote by $\sigma$ the transport density between $f^+$ and $f^-$. Then, the restriction of $\sigma$ to $\Delta_n^{\prime}$ is the transport density $\sigma_n$ between $f_n^+:=1_{\Delta_n^{\prime}}$ and $f_n^-$. Indeed, for all $n \in \mathbb{N}^*$, if $T_n$ is an optimal transport map between $f_n^\pm\,$ and if $\,u_n$ is the corresponding Kantorovich potential such that $u_n(-1,0)=0$, for all $n\in \mathbb{N}^*$, then it is not difficult to check that \\
 $$ T(x):=T_n(x),\;\;\mbox{for a.e}\,\;x \in \Delta_n^{\prime}$$\\
 is an optimal transport map between $f^\pm$
 and the corresponding Kantorovich potential will be\\
 $$ u(x):=u_n(x),\;\;\mbox{for all}\,\;x \in \Delta_n^{\prime}.$$\\
 By (\ref{definition de la density}), we infer that the restriction of $\sigma$ to $\Delta_n^{\prime}$ is $\sigma_n$. Yet, by Section \ref{Sec.3}, we have already shown that\\
  $$ |\nabla \sigma_n| \approx \frac{1}{J_n}$$\\
  where $J_n$ is defined as in (\ref{Jacobien}) on $\Delta_n$. Hence, 
 \begin{eqnarray*} 
 \sum_{n=1}^{\infty} ||\nabla \sigma_n||_{L^1(\Delta_n^{\prime})} & \approx &  \sum_{n=1}^{\infty} \int_0^{l_n}\int_0^{\delta}|\nabla \sigma_n(t,a)|J_n(t,a)\,\mathrm{d}t\,\mathrm{d}a\\ \\
 &\approx &  \sum_{n=1}^{\infty} l_n
= \sum_{n=1}^{\infty} \frac 1n=+\infty \end{eqnarray*}\\ 
 where $\delta>0$ small enough. Hence, the transport density $\sigma \notin BV(\Omega)$. On the other hand, we will show that the target mass $f^-$ is in $BV(\mathbb{R}^2)$. Using (\ref{Holder for tiers eta}), it is easy to prove that\\
$$\sum_{n=1}^{\infty} ||\nabla f_n^-||_{L^1(\Delta_n^{\prime})} \leq C \sum_{n=1}^{\infty}(l_n^\gamma+l_n^2)<+\infty.$$\\
\\
In addition, for a fixed $n \in \mathbb{N}^*$ and after a suitable roto-translation of axis so that $\Delta_n^{\prime}=\Delta_n$, we can assume that
$$ f_n^-(x_1,x_2)=1+\beta (\zeta^{\prime\prime}(x_1) +\eta^{\prime\prime}(|x_2|))$$
and

 $$ f_{n+1}^-(x_1,x_2)=1\,+\,\beta(\zeta^{\prime\prime}(\cos(\theta_{n+1})(x_1 + l_n) \,+\, \sin(\theta_{n+1}) x_2 \,- \,l_{n+1}) \,+\, \eta^{\prime\prime}(|-\sin(\theta_{n+1})(x_1 + \,l_n) +\, \cos(\theta_{n+1}) x_2|))$$\\
 where $$\theta_{n+1}:=\alpha_n + \alpha_{n+1} \approx l_{n+1}^{\gamma}.$$\\
Hence, it is not difficult to check that
\\
$$\bigg|f_{n+1}^-\bigg(x_1,\frac{l_n^{\gamma}}{2}(x_1+l_n)\bigg)-f_n^-\bigg(x_1,\frac{l_n^{\gamma}}{2}(x_1+l_n)\bigg)\bigg| \leq  C (l_n^\gamma+l_n^2).$$\\
 Finally, we get
$$ \sum_{n=1}^{\infty}\int_{\partial\Delta_{n} \cap \partial\Delta_{n+1}}|f_{n+1}^-(z)-f_n^-(z)|\mathrm{d}z \leq C \sum_{n=1}^{\infty}(l_n^\gamma+l_n^2)<+\infty.$$
 \\ \\
As $\,f^-$ is bounded and $\,Per(\Omega) \approx \sum\limits_{n} (l_n^\gamma+l_n^2)<+\infty$, we infer that the target mass $f^- \in BV(\mathbb{R}^2)$ and the statement (1.4) follows.
\section{Counter-examples with compactly supported smooth densities on the whole plane} \label{Sec.5}
In this section, we want to show that is also possible to construct the target measure $f^-$ so that it will be regular on $\mathbb{R}^2$. Firstly, let us observe that the function $\zeta$ (see Section \ref{sec.2}) can be replaced by $\psi\zeta$, where $\psi$ is a $C^\infty$ function such that $\psi=1$ on $[-1,1-\varepsilon^\prime]$ and $\psi=0$ on $ [1-\varepsilon,1]$, ($0<\varepsilon<\varepsilon^{\prime}<1$). Let $\chi_1,\,\chi_2$ be two cutoff functions supported on $\Delta \cup R(\Delta)$, where $R$ is the reflection map with respect to the $x_1$-axis, 
such that $\spt(\chi_2) \subset \{\chi_1 =1\}$, 
$\Delta_{a_0} \cap \{x : x_1 \leq 1-\varepsilon\} \subset \{\chi_1=1\}$ (where $a_0 \in (\varepsilon,1)$ is such that $\spt(\chi_2) \subset \Delta_{a_0} \cup R(\Delta_{a_0})$), $\Delta_{\varepsilon} \cap \{x : x_1 \leq 1-\varepsilon\} \subset \{\chi_2=1\}$
and $\chi_1,\,\chi_2$ are symmetric with respect to the $x_1$-axis. Set,\\
$$f^+:={\chi}_1$$ 
and $$f^-:=\chi_1+\beta(((\psi \zeta)^{\prime\prime}(x_1)+\eta^{\prime\prime}(|x_2|))\chi_2 + \varphi(x_1)c(a(x_1,|x_2|))),$$\\
where $\varphi$ is a non-negative $C^\infty$ function such that $\spt(\varphi) \subset (1-\varepsilon^\prime,1-\varepsilon)$
and $c$ is to be determined in such a way that
\begin{equation*} 
  \int_{\Delta_a} f^+ = \int_{\Delta_a} f^- \;\;\;\;\mbox{for all}\;a \in (0,1),
  \end{equation*}
  which is equivalent to say that 
  
  $$-\int_{\Delta_a}((\psi \zeta)^{\prime\prime}(x_1)+\eta^{\prime\prime}(x_2))\chi_2(x_1,x_2)  = \int_{\Delta_a}  \varphi(x_1)c(a(x_1,x_2)), \;\;\;\;\mbox{for all}\;a \in (0,1).$$\\ \\
 Differentiating this equality with respect to $a$, we get 
  
  $$c(a)=\frac{-\int_{-a}^1 (\gamma(x_1+a)+a)((\psi \zeta)^{\prime\prime}(x_1)+\eta^{\prime\prime}(\frac{a^\gamma}{2}(x_1+a)))\chi_2(x_1,\frac{a^\gamma}{2}(x_1+a))\,\mathrm{d}x_1}{\int_{-a}^1 (\gamma(x_1+a)+a)\varphi(x_1)\,\mathrm{d}x_1}, \;\;\;\;\mbox{for all}\;a \in (0,1).$$\\
 By (\ref{Existence of a transport map}), we have
 
 $$
 - \int_{-a}^1 (\gamma(x_1+a)+a)(\psi \zeta)^{\prime\prime}(x_1)\,\mathrm{d}x_1 =  \int_{-a}^1 (\gamma(x_1+a)+a)\eta^{\prime\prime}\bigg(\frac{a^\gamma}{2}(x_1+a)\bigg)\,\mathrm{d}x_1, \;\;\;\;\mbox{for all}\;a \in (0,1).$$ \\ \\
 Hence, for $a<\varepsilon$, we get
 
 $$c(a)=\frac{\int_{-a}^1 (\gamma(x_1+a)+a)\eta^{\prime\prime}(\frac{a^\gamma}{2}(x_1+a))(1-\chi_2(x_1,\frac{a^\gamma}{2}(x_1+a)))\,\mathrm{d}x_1}{\int_{-a}^1 (\gamma(x_1+a)+a)\varphi(x_1)\,\mathrm{d}x_1}$$
 and \\$$c^\prime(a)=\frac{1}{\int_{-a}^1 (\gamma(x_1+a)+a)\varphi(x_1)\,\mathrm{d}x_1}\bigg((\gamma+1)\int_{-a}^1 \eta^{\prime\prime}\bigg(\frac{a^\gamma}{2}(x_1+a)\bigg)\bigg(1-\chi_2\bigg(x_1,\frac{a^\gamma}{2}(x_1+a)\bigg)\bigg)\,\mathrm{d}x_1$$\\
 $$+\frac{a^{\gamma-1}}{2}\int_{-a}^1 (\gamma(x_1+a)+a)^2\,\eta^{\prime\prime\prime}\bigg(\frac{a^\gamma}{2}(x_1+a)\bigg)\bigg(1-\chi_2\bigg(x_1,\frac{a^\gamma}{2}(x_1+a)\bigg)\bigg)\,\mathrm{d}x_1 - (\gamma+1)\bigg(\int_{-a}^1 \varphi(x_1)\,\mathrm{d}x_1\bigg) c(a)$$\\
 $$-\frac{a^{\gamma-1}}{2}\int_{-a}^1 (\gamma(x_1+a)+a)^2\,\eta^{\prime\prime}\bigg(\frac{a^\gamma}{2}(x_1+a)\bigg)\partial_{x_2}\chi_2\bigg(x_1,\frac{a^\gamma}{2}(x_1+a)\bigg)\,\mathrm{d}x_1\bigg).$$\\  \\
 By (\ref{Holder for tiers eta}), we infer that 
 $$c^\prime(a) \leq C a$$
 and
 $$||\nabla (\varphi\,c(a))||_{L^p(\Delta)}^p \approx \int_{1-\varepsilon^\prime}^{1-\varepsilon}\int_0^{\varepsilon} \bigg(\varphi(x_1)^p\frac{|c^\prime(a)|^p}{J(t,a)^p}+|\nabla\varphi(x_1)|^p|c(a)|^p\bigg)\,\mathrm{d}x_2\,\mathrm{d}x_1 \approx \int_0^\varepsilon \frac{1}{a^{1-(\gamma-(\gamma-2)p)}}\,\mathrm{d}a.$$\\ 
Hence, for $\gamma>2$,
 $$ f^- \in W^{1,\frac{\gamma}{\gamma-2}-\varepsilon}(\mathbb{R}^2),\;\;\mbox{for all}\,\,\varepsilon>0.$$\\
Similarly, we get that for $\gamma=\frac{1}{2}$: $f^- \in C^\infty(\mathbb{R}^2)$, for $\gamma=1$: $f^- \in C^{2,1}(\mathbb{R}^2)$, for $1<\gamma<2$: $f^- \in C^{1,\frac{2}{\gamma}-1}(\mathbb{R}^2)$ and, finally, for $\gamma=2$: $f^- \in C^{0,1}(\mathbb{R}^2)$.\\ \\
{\bf Acknowledgments:} the author would like to thank Prof. Filippo Santambrogio for interesting discussions and suggestions. The author also acknowledges the support of the ANR project ANR-12-BS01-0014-01 GEOMETRYA.


\begin{thebibliography}{99}
 \bibitem{1}{\sc L. Ambrosio,}
  \newblock  {Lecture notes on optimal transport problems}, in       \newblock {\it Mathematical aspects of evolving
interfaces}, Lecture Notes in Mathematics (1812) (Springer, New York, 2003), pp. 1-52.
\bibitem{BouButJEMS}  {\sc G. Bouchitt\'{e}, G. Buttazzo}, Characterization of optimal shapes and masses through Monge-Kantorovich equation,
{\it J. Eur. Math. Soc.}
 3 (2), 139-168, 2001.
\bibitem{Buttazo}
{\sc G. Buttazzo, E. Stepanov,}
 On Regularity of Transport Density in the Monge-Kantorovich Problem,
{\it SIAM Journal on Control and Optimization } 42 (3): 1044-1055.
   \bibitem{CafFelMcC} {\sc L. Caffarelli, M. Feldman, R. McCann,}
 Constructing optimal maps for Monge's transport problem as a limit of strictly convex costs,
{\it Journal of the American Mathematical Society} 15 (1), 1-26, 2002.
 \bibitem{Champion} {\sc T. Champion, L. De Pascale,}
\newblock{The Monge problem in $\mathbb{R}^d$},
 {\it Duke Math. J.} 157, 3 (2011), 551-572.
\bibitem{Colombo}{ \sc M. Colombo, E. Indrei,}
\newblock{Obstructions to regularity in the classical Monge
problem,}
\newblock {\it{Math. Res. Lett.}} 21 (2014), 697-712.
 \bibitem{55}{\sc L. De Pascale, L. C. Evans, A. Pratelli,}  \newblock Integral estimates for transport densities, 
    \newblock {\it Bull. of the London Math. Soc.} 36, n. 3, pp. 383-395, 2004.
 
  \bibitem{66}{\sc L. De Pascale, A. Pratelli,}
   \newblock Regularity properties for Monge Transport Density and for Solutions of some Shape Optimization Problem,
   \newblock {\it Calc. Var. Par. Diff. Eq}. 14, n. 3, pp. 249-274, 2002.
 
  \bibitem{77}{\sc L. De Pascale, A. Pratelli,}
  \newblock  Sharp summability for Monge Transport density via  Interpolation,
   \newblock{\it ESAIM Control Optim. Calc. Var.} 10, n. 4, pp. 549-552, 2004.
\bibitem{Gangbo}{\sc L. C. Evans, W. Gangbo,}
\newblock Differential equations methods for the Monge-Kantorovich mass
transfer problem,
\newblock {{\it Mem. Amer. Math. Soc.} 137 (1999), no. 653.}
 \bibitem{33}{\sc M. Feldman, R. McCann,}
 \newblock Uniqueness and transport density in Monge's mass transportation problem,
 \newblock{{\it Calc. Var. Par. Diff. Eq.} 15, n. 1, pp. 81-113, 2002.}
  \bibitem{Pratelli}{\sc I. Fragal\`a, M.S. Gelli, A. Pratelli,}\newblock{ Continuity of an optimal transport in Monge problem,}
\newblock{{\it J. Math. Pures Appl.} 84 (2005), 1261-1294.}
\bibitem{Wang2016}{\sc Q.R. Li, F. Santambrogio, X.J. Wang,} \newblock{Continuity for the Monge mass transfer problem in two dimensions,} {cvgmt/paper/3219}.
 \bibitem{Wang}{\sc Q.R. Li, F. Santambrogio, X.J. Wang,} \newblock{Regularity in Monge's mass transfer problem,}
 \newblock{\it J. Math. Pures Appl.} 102 (6), 1015-1040 (2014).
 \bibitem{Kanto} {\sc L. Kantorovich,}
 \newblock {On the transfer of masses,
{\it Dokl. Acad. Nauk. USSR}, (37), 7-8, 1942.}
 \bibitem{Monge}{\sc G. Monge,}
 \newblock{M\'emoire sur la th\'eorie des d\'eblais et des remblais,}
 \newblock{\it Histoire de l'Acad\'emie Royale
des Sciences de Paris} (1781), 666-704.
\bibitem{Sudakov}{\sc V.N. Sudakov,}
\newblock Geometric problems in the theory of infinite-dimensional probability distributions,
 \newblock {\it Proc. Stekelov Inst. Math.} 141 (1979), 1-178.
  \bibitem{9}{\sc F. Santambrogio,}
\newblock  Absolute continuity and summability of transport densities: simpler proofs and new estimates,
\newblock {\it Calc. Var. Par. Diff. Eq.} (2009) 36: 343-354.

\bibitem{8} {\sc F. Santambrogio,} {\it Optimal Transport for Applied Mathematicians} in {\it Progress in Nonlinear Differential Equations and Their Applications}  87, Birkh\"auser
Basel (2015).
\bibitem{wang}{\sc N.S. Trudinger, X.J. Wang,}
\newblock  On the Monge mass transfer problem,
\newblock {\it Calc. Var. PDE} 13
(2001), 19-31.


\bibitem{11}{ \sc C. Villani,}
 \newblock {\it Topics in Optimal Transportation.} Graduate Studies in Mathematics. Vol. 58, 2003.
\end{thebibliography}
\end{document}